\documentclass[a4paper , 11pt]{amsart}

\usepackage[top=30truemm,bottom=35truemm,left=30truemm,right=30truemm]{geometry}
\usepackage{amsmath}
\usepackage{amsthm}
\usepackage{amsfonts}
\usepackage{bm}

\newtheorem{thm}{Theorem}[section]
\newtheorem{lem}[thm]{Lemma}

\newtheorem{cor}[thm]{Corollary}
\theoremstyle{definition}
\newtheorem{prob}[thm]{Problem}
\newtheorem*{rem}{Remark}

\newcommand{\abs}[1]{\left\lvert#1\right\rvert} 

\makeatletter 
\@tempcnta\z@
\loop\ifnum\@tempcnta<26
\advance\@tempcnta\@ne
\expandafter\edef\csname f\@Alph\@tempcnta\endcsname{\noexpand\mathfrak{\@Alph\@tempcnta}}
\expandafter\edef\csname l\@Alph\@tempcnta\endcsname{\noexpand\mathbb{\@Alph\@tempcnta}}
\expandafter\edef\csname c\@Alph\@tempcnta\endcsname{\noexpand\mathcal{\@Alph\@tempcnta}}
\repeat

\title{A note on quadratic approximation for Liouville numbers}
\author[T.~Ooto]{Tomohiro Ooto}
\address{Tecnos Data Science Engineering, Inc., 27F., Tokyo Opera City, 3-20-2, Nishishinjuku, Shinjuku-ku, Tokyo, 163-1427, Japan}
\email{ooto.tomohiro@gmail.com}

\begin{document}

\begin{abstract}
	Schleischitz [arXiv:1701.01129] determined exponents of best approximations to a strong Liouville number by integer polynomials and algebraic numbers of precribed degree.
	In this note, we show that we cannot extend his result to arbitrary Liouville numbers.
\end{abstract}

\maketitle

	
\section{Introduction}\label{sec_introduction}

Let $\xi$ be a real number and $n \geq 1$ be an integer.
We denote by $w_n (\xi)$ the supremum of the real numbers $w$ such that the system
\begin{equation}\label{eq_wnreal}
	H(P) \leq H, \quad 0 < \abs{P(\xi)} \leq H ^ {- w}
\end{equation}
has a solution $P(X) \in \lZ[X]$ of degree at most $n$, for arbitrarily large $H$.
Here, $H(P)$ is defined to be the maximum of the absolute values of the coefficients of $P(X)$.
We denote by $w_n ^ {*} (\xi)$ the supremum of the real numbers $w ^ {*}$ such that the system
\begin{equation}\label{eq_wn*real}
	H(\alpha) \leq H, \quad 0 < \abs{\xi - \alpha} \leq H(\alpha) ^ {- 1} H ^ {- w}
\end{equation}
has a solution real algebraic number $\alpha$ of degree at most $n$, for arbitrarily large $H$.
Here, $H(\alpha)$ is equal to $H(P)$, where $P(X)$ is the minimal polynomial of $\alpha$ over $\lZ$.
These functions $w_n$ and $w_n ^ {*}$ are called \textit{Diophantine exponents} and introduced by Mahler \cite{Mahler32} and Koksma \cite{Koksma39}, respectively.
Recently, Schleischitz \cite{Schleischitz17} introduced new Diophantine exponents $w_{= n}$ and $w_{= n} ^ {*}$.
Let $w_{= n} (\xi)$ be the supremum of the real numbers $w$ such that \eqref{eq_wnreal} has an irreducible solution $P(X) \in \lZ[X]$ of degree $n$, for arbitrarily large $H$.
Similarly, let $w_{= n} ^ {*} (\xi)$ be the supremum of the real numbers $w ^ {*}$ such that \eqref{eq_wn*real} has a solution real algebraic number $\alpha$ of degree $n$, for arbitrarily large $H$.
It is obvious that, for any real number $\xi$,
\begin{equation}\label{eq_w_1}
	w_1(\xi) = w_1 ^{*}(\xi) = w_{= 1} (\xi) = w_{= 1} ^ {*} (\xi).
\end{equation}

We recall that a real number $\xi$ is called a \textit{Liouville number} if $w_1(\xi) = + \infty$.
Let $[a_0, a_1, a_2, \ldots ]$ be the continued fraction of $\xi$.
We recall that $\xi$ is called a \textit{strong Liouville number} introduced by LeVeque \cite{LeVeque53} if
\begin{equation*}
	\lim_{n \rightarrow \infty} \frac{\log a_{n + 1}}{\log a_n} = + \infty.
\end{equation*}

Schleischitz \cite{Schleischitz17} investigated strong Liouville numbers as follows.

\begin{thm}\cite[Corollary 3.12]{Schleischitz17}
	Let $\xi$ be a strong Liouville number.
	Then we have
	\begin{equation*}
		w_{= n} (\xi) = w_{= n} ^ {*} (\xi) = n
	\end{equation*}
	for any integer $n \geq 2$.
\end{thm}

He cannot expect Theorem \ref{thm_main} to extend for arbitrary Liouville numbers (see \cite[p.~12]{Schleischitz17}).
In this note, we answer this question using by the theory of continued fractions.

We recall quasi-periodic continued fractions.
Let $(a_n)_{n \geq 0}$ be a non-ultimately periodic sequence of positive integers.
Let $(n_k)_{k \geq 0}$ be an increasing sequence of positive integers and $(\lambda_k)_{k \geq 0}, (r_k)_{k \geq 0}$ be sequences of positive integers.
Assume that, for any integer $k \geq 0$, we have $n_{k + 1} \geq n_k + \lambda_k r_k$ and
\begin{equation*}
	a_{m + r_k} = a_m\ \text{for}\ n_k \leq m \leq n_k + (\lambda_k - 1) r_k -1.
\end{equation*}
Then, $\xi = [a_0, a_1, a_2, \ldots ]$ is called a \textit{quasi-periodic continued fraction}.
A transcendence and transcendence measures of quasi-periodic continued fractions is studied by many authors (see e.g.\ \cite{Adamczewski07, Baker62, Bugeaud12, Maillet06}).

Throughout this note, for integers $n, m \geq 1$, $\overline{a_1, \ldots, a_m} ^ n$ (resp.\ $\overline{a_1, \ldots, a_m}$) means that the $m$ letters $a_1, \ldots, a_m$ is repeated $n$ times (resp.\ infinitely many times).

Our main theorem is the following.

\begin{thm}\label{thm_main}
	Let $\boldsymbol{b} = (b_n)_{n \geq 0}$ and $\boldsymbol{\lambda} = (\lambda_n)_{n \geq 0}$ be sequences of positive integers.
	We consider a quasi-periodic continued fraction
	\begin{equation*}
		\xi (\boldsymbol{b}, \boldsymbol{\lambda}) 
		= [0, \overline{b_0} ^ {\lambda_0}, \overline{b_1} ^ {\lambda_1}, \ldots].
	\end{equation*}
	Assume that
	\begin{equation}\label{eq_main_cond}
		\lim_{n \rightarrow \infty} b_n
		= \lim_{n \rightarrow \infty} \lambda_n
		= \lim_{n \rightarrow \infty} \frac{\log b_{n + 1}}{\sum_{k = 0}^{n} \lambda_k \log b_k}
		= + \infty.
	\end{equation}
	Then $\xi (\boldsymbol{b}, \boldsymbol{\lambda})$ is a Liouville number and satisfies
	\begin{equation*}
		w_{= 2} (\xi (\boldsymbol{b}, \boldsymbol{\lambda}))
		= w_{= 2} ^ {*} (\xi (\boldsymbol{b}, \boldsymbol{\lambda}))
		= + \infty.
	\end{equation*}
\end{thm}

\begin{rem}
	Let $\xi (\boldsymbol{b}, \boldsymbol{\lambda}) = [0, a_1, a_2, \ldots ]$ be a continued fraction as in Theorem \ref{thm_main}.
	It is easily seen that
	\begin{equation*}
		\liminf_{n \rightarrow \infty} \frac{\log a_{n + 1}}{\log a_n} \leq 1.
	\end{equation*}
	Therefore, $\xi (\boldsymbol{b}, \boldsymbol{\lambda})$ is not a strong Liouville number.
\end{rem}

\begin{cor}\label{cor_main}
	There exist uncountablly many Liouville numbers $\xi$ such that
	\begin{equation*}
		w_{= 2} (\xi) = w_{= 2} ^ {*} (\xi) = + \infty .
	\end{equation*}
\end{cor}

We address the following problem arising from Corollary \ref{cor_main}.

\begin{prob}
	Does there exists a Liouville number $\xi$ such that
	\begin{equation*}
		w_{= n} (\xi) = w_{= n} ^ {*} (\xi) = + \infty
	\end{equation*}
	for any integer $n \geq 2$?
\end{prob}

This paper is organized as follows.
In Section \ref{sec_pre}, we prepare some lemmas for the proof of the main results.
In Section \ref{sec_proofs}, we prove Theorem \ref{thm_main} and Corollary \ref{cor_main}.

\section{Preliminaries}\label{sec_pre}

Let $\xi = [a_0, a_1, a_2, \ldots]$ be a real number.
We define sequences $(p_n)_{n \geq - 1}$ and $(q_n)_{n \geq - 1}$ by
\begin{equation*}
	\begin{cases}
		p_{- 1} = 1,\ p_0 = a_0,\ p_n = a_n p_{n - 1} + p_{n - 2},\ n \geq 1,\\
		q_{- 1} = 0,\ q_0 = 1,\ q_n = a_n q_{n - 1} + q_{n - 2},\ n\geq 1.
	\end{cases}
\end{equation*}
We call $(p_n / q_n)_{n \geq 0}$ \textit{the convergent sequence of $\xi$}.

The following lemma is well-known result.

\begin{lem}\label{lem_rat_app_koushiki}
	Let $\xi = [a_0, a_1, a_2, \ldots ]$ be a real number and $(p_n / q_n)_{n \geq 0}$ be the convergent sequence of $\xi$.
	Then we have
	\begin{equation*}
		w_1 (\xi) 
		= 1 + \limsup_{n \rightarrow \infty} \frac{\log a_{n + 1}}{\log q_n}.
	\end{equation*}
\end{lem}

\begin{lem}\label{lem_new_exponent_hikaku}
	Let $\xi$ be a real number and $n \geq 1$ be an integer.
	Then we have
	\begin{equation*}
		w_{= n} ^ {*} (\xi) \leq w_{= n} (\xi).
	\end{equation*}
\end{lem}

\begin{proof}
	See \cite[p.~4]{Schleischitz17}.
\end{proof}

We ready elementarily three lemmas on continued fractions (see e.g.\ \cite{Perron29} for a proof).

\begin{lem}\label{lem_cf_hikaku}
	Let $\xi = [a_0, a_1, a_2, \ldots]$ and $\zeta = [b_0, b_1, b_2, \ldots]$ be real numbers.
	Assume that there exist integer $k \geq 1$ such that $a_j = b_j$ for any $0 \leq j \leq k$.
	Then we have
	\begin{equation*}
		\abs{\xi - \zeta} \leq q_k ^{- 2},
	\end{equation*}
	where $(p_n/q_n)_{n \geq 0}$ is the convergent sequence of $\xi$.
\end{lem}

\begin{lem}\label{lem_hight_upper}
	Let $\xi$ be a quadratic real number with ultimately periodic continued fraction
	\begin{equation*}
		\xi = [0, a_1, \ldots, a_r, \overline{a_{r + 1}, \ldots, a_{r + s}}].
	\end{equation*}
	Then we have $H(\xi) \leq 2 q_r q_{r + s}$, where $(p_n / q_n)_{n \geq 0}$ is the convergent sequence of $\xi$.
\end{lem}

For positive integers $a_1, \ldots, a_n$, denote by $K_n (a_1, \ldots, a_n)$ the denominator of the rational number $[0, a_1, \ldots, a_n]$.
It is commonly called a \textit{continuant}.

\begin{lem}\label{lem_continua}
	For any positive integers $a_1, \ldots, a_n$ and any integer $k$ with $1 \leq k \leq n - 1$, we have
	\begin{align*}
		K_k (a_1, \ldots, a_k) & K_{n - k} (a_{k + 1}, \ldots, a_n) \\
		& \leq K_n (a_1, \ldots, a_n)
		\leq 2 K_k (a_1, \ldots, a_k) K_{n - k} (a_{k + 1}, \ldots, a_n), \\
		K_n (a_1, \ldots, a_n)
		& \leq (1 + \max \{ a_1, \ldots, a_n \}) ^ n.
	\end{align*}
\end{lem}

\begin{lem}\label{lem_continua_upp_low}
	Let $n \geq 0$ be an integer and $b_0, \ldots, b_n, \lambda_0, \ldots, \lambda_n \geq 1$ be integers.
	We put $m(n) := \sum_{k = 0}^{n} \lambda_k$.
	Then we have
	\begin{equation*}
		\sum_{k = 0}^{n} \lambda_k \log b_k
		\leq \log K_{m(n)} (\overline{b_0} ^ {\lambda_0}, \ldots, \overline{b_n} ^ {\lambda_n})
		\leq 2 \sum_{k = 0}^{n} \lambda_k \log (b_k + 1).
	\end{equation*} 
\end{lem}

\begin{proof}
	By Lemma \ref{lem_continua}, we have
	\begin{equation*}
		\log K_{m(n)} (\overline{b_0} ^ {\lambda_0}, \ldots, \overline{b_n} ^ {\lambda_n})
		\geq \sum_{k = 0}^{n} \lambda_k \log K_1 (b_k)
		= \sum_{k = 0}^{n} \lambda_k \log b_k.
	\end{equation*}
	By Lemma \ref{lem_continua}, we also obtain
	\begin{align*}
		\log K_{m(n)} (\overline{b_0} ^ {\lambda_0}, \ldots, \overline{b_n} ^ {\lambda_n})
		& \leq n \log 2 + \sum_{k = 0}^{n} \log K_{\lambda_k} (\overline{b_k} ^ {\lambda_k}) \\
		& \leq n \log 2 + \sum_{k = 0}^{n} \lambda_k \log (b_k + 1) \\
		& \leq 2 \sum_{k = 0}^{n} \lambda_k \log (b_k + 1).
	\end{align*}
\end{proof}

\section{Proof of Main results}\label{sec_proofs}

\begin{proof}[Proof of Theorem \ref{thm_main}]
	For simplicity of notation, we put $\xi := \xi (\boldsymbol{b}, \boldsymbol{\lambda})$.
	Let $(p_n / q_n)_{n \geq 0}$ be the convergent sequence of $\xi$.
	For an integer $n \geq 0$, we put $m(n) := \sum_{k = 0}^{n} \lambda_k$.
	By the assumption, there exists an integer $N_0 \geq 0$ such that $b_n \geq 2$ for any integer $n \geq N_0$.
	Then, by Lemma \ref{lem_continua_upp_low}, we obtain
	\begin{align}
		\log q_{m(n)}
		& = \log K_{m(n)} (\overline{b_0} ^ {\lambda_0}, \ldots, \overline{b_n} ^ {\lambda_n})
		\leq 2 \sum_{k = 0}^{n} \lambda_k \log (b_k + 1) \nonumber \\
		& \leq 4 \sum_{k = 0}^{n} \lambda_k \log b_k + \sum_{k = 0}^{N_0} \lambda_k \log (4 / b_k) \label{eq_thm_w_1}
	\end{align}
	for any integer $n \geq N_0$.
	Therefore, it follows from Lemma \ref{lem_rat_app_koushiki} that
	\begin{align*}
		w_1 (\xi)
		& \geq 1 + \limsup_{n \rightarrow \infty} \frac{\log a_{m(n) + 1}}{\log q_{m(n)}} \\
		& \geq \limsup_{n \rightarrow \infty} \frac{\log b_{n + 1}}{4 \sum_{k = 0}^{n} \lambda_k \log b_k + \sum_{k = 0}^{N_0} \lambda_k \log (4 / b_k)}
		= + \infty.
	\end{align*}
	By \eqref{eq_w_1}, we obtain $w_{= 1} (\xi) = w_{= 1} ^ {*} (\xi) = + \infty$.
	For an integer $n \geq 0$, we consider a quadratic real number
	\begin{equation*}
		\alpha_n = [0, \overline{b_0} ^ {\lambda_0}, \ldots, \overline{b_n} ^ {\lambda_n}, \overline{b_{n + 1}}].
	\end{equation*}
	It follows from the assumption, \eqref{eq_thm_w_1}, Lemma \ref{lem_hight_upper} and \ref{lem_continua} that, for all sufficiently large $n$,
	\begin{align*}
		\log H(\alpha_n)
		& \leq \log (2 q_{m(n)} q_{m(n) + 1}) \\
		& \leq 2 \log K_{m(n)} (\overline{b_1} ^ {\lambda_1}, \ldots, \overline{b_n} ^ {\lambda_n}) + \log K_1 (b_{n + 1}) + \log 4 \\
		& \leq 8 \sum_{k = 0}^{n} \lambda_k \log b_k + \log b_{n + 1} + \log 4 + 2 \sum_{k = 0}^{N_0} \lambda_k \log (4 / b_k) \\
		& \leq 3 \log b_{n + 1}.
	\end{align*}
	By Lemma \ref{lem_cf_hikaku}, \ref{lem_continua}, and \ref{lem_continua_upp_low}, we obtain
	\begin{align*}
		- \log \abs{\xi - \alpha_n}
		& \geq 2 \log q_{m(n + 1)} \\
		& \geq 2 \log K_{\lambda_{n + 1}} (\overline{b_{n + 1}} ^ {\lambda_{n + 1}}) \\
		& \geq 2 \lambda_{n + 1} \log b_{n + 1} 
	\end{align*}
	for any integer $n \geq 0$.
	Therefore, we have
	\begin{equation*}
		0 < \abs{\xi - \alpha_n} \leq H(\alpha_n) ^ {- 2 \lambda_{n + 1} / 3}
	\end{equation*}
	for all sufficiently large $n$.
	Thus, by Lemma \ref{lem_new_exponent_hikaku}, we obtain $w_{= 2} (\xi) = w_{= 2} ^ {*} (\xi) = \infty$.
\end{proof}

\begin{proof}[Proof of Corollary \ref{cor_main}]
	Let $\boldsymbol{\varepsilon} = (\varepsilon_n)_{n \geq 0}$ and $\boldsymbol{\lambda} = (\lambda_n)_{n \geq 0}$ be integer sequences with $\varepsilon_n \in \{ 1, 2 \}$ and $\lambda_n = n$ for any integer $n \geq 0$.
	We define a sequence $\boldsymbol{b}(\boldsymbol{\varepsilon}) = (b(\boldsymbol{\varepsilon})_n)_{n \geq 0}$ by $b(\boldsymbol{\varepsilon})_0 = b(\boldsymbol{\varepsilon})_1 = 2$ and
	\begin{equation*}
		b(\boldsymbol{\varepsilon})_n = \prod_{k = 0}^{n - 1} b(\boldsymbol{\varepsilon})_k ^ {k n \varepsilon_n}
	\end{equation*}
	for any integer $n \geq 2$.
	It is easily seen that $(\boldsymbol{b}(\boldsymbol{\varepsilon}), \boldsymbol{\lambda})$ satisfies \eqref{eq_main_cond}.
	Therefore, it follows from Theorem \ref{thm_main} that $\xi (\boldsymbol{b}(\boldsymbol{\varepsilon}), \boldsymbol{\lambda})$ is a Liouville number and
	\begin{equation*}
		w_{= 2} (\xi (\boldsymbol{b}(\boldsymbol{\varepsilon}), \boldsymbol{\lambda}))
		= w_{= 2} ^ {*} (\xi (\boldsymbol{b}(\boldsymbol{\varepsilon}), \boldsymbol{\lambda}))
		= + \infty .
	\end{equation*}
	Let $\boldsymbol{\varepsilon}' = (\varepsilon_n')_{n \geq 0}$ be an integer sequence with $\boldsymbol{\varepsilon} \neq \boldsymbol{\varepsilon}'$ and $\varepsilon_n' \in \{ 1, 2 \}$ for any integer $n \geq 0$.
	We put an integer $N := \min \{ n \mid \varepsilon_n \neq \varepsilon_n' \}$ and
	\begin{gather*}
		\xi (\boldsymbol{b}(\boldsymbol{\varepsilon}), \boldsymbol{\lambda})
		=: [0, a_1(\boldsymbol{\varepsilon}), a_2(\boldsymbol{\varepsilon}), \ldots ], \\
		\xi (\boldsymbol{b}(\boldsymbol{\varepsilon}'), \boldsymbol{\lambda})
		=: [0, a_1(\boldsymbol{\varepsilon}'), a_2(\boldsymbol{\varepsilon}'), \ldots ].
	\end{gather*}
	For an integer $n \geq 0$, we put $m(n) := \sum_{k = 0}^{n} \lambda_k$.
	For a convenience, we put $m(- 1) := 0$.
	By $a_{m(N - 1) + 1} (\boldsymbol{\varepsilon}) \neq a_{m(N - 1) + 1} (\boldsymbol{\varepsilon}')$, we obtain $\xi (\boldsymbol{b}(\boldsymbol{\varepsilon}), \boldsymbol{\lambda}) \neq \xi (\boldsymbol{b}(\boldsymbol{\varepsilon}'), \boldsymbol{\lambda})$.
	Since uncountably many choices of such sequences $\boldsymbol{\varepsilon}$, the proof is complete.
\end{proof}

\end{document}